\documentclass[11pt]{article}
\usepackage{amsmath,amsthm}
\usepackage{amssymb,mathrsfs,pifont}
\usepackage{bm,mathtools}

\usepackage{xcolor}
\usepackage{geometry}
\usepackage[colorlinks=true,
linkcolor=blue,citecolor=blue,
urlcolor=blue,pagebackref]{hyperref}

\makeatletter
\def\@seccntDot{.}
\def\@seccntformat#1{\csname the#1\endcsname\@seccntDot\hskip 0.5em}
\renewcommand\section{\@startsection{section}{1}{\z@}%
{18\p@ \@plus 6\p@ \@minus 3\p@}%
{9\p@ \@plus 6\p@ \@minus 3\p@}%
{\large\bfseries\boldmath}}
\renewcommand\subsection{\@startsection{subsection}{2}{\z@}%
{15\p@ \@plus 6\p@ \@minus 3\p@}%
{6\p@ \@plus 6\p@ \@minus 3\p@}%
{\itshape}}
\renewcommand\subsubsection{\@startsection{subsubsection}{3}{\z@}%
{12\p@ \@plus 6\p@ \@minus 3\p@}%
{\p@}%
{}}
\makeatother

\usepackage{microtype}

\theoremstyle{plain}
\newtheorem{theorem}{Theorem}[section]
\newtheorem{lemma}{Lemma}[section]
\newtheorem{problem}{Problem}[section]
\newtheorem{corollary}{Corollary}[section]

\theoremstyle{definition}

\newtheorem{claim}{Claim}[section]

\numberwithin{equation}{section}
\allowdisplaybreaks
\DeclareMathOperator{\vol}{Vol}

\begin{document}

\title{Maximum spread of $K_r$-minor free graphs\thanks{Supported by National Natural 
Science Foundation of China (12171002, 12331012, 12471320), Excellent University 
Research and Innovation Team in Anhui Province (2024AH010002), and 
Anhui Provincial Natural Science Foundation (2408085Y003).} }
\author{Wenyan Wang, Lele Liu, Yi Wang\thanks{Corresponding author: wangy@ahu.edu.cn} \\
{\small  \it School of Mathematical Sciences, Anhui University, Hefei 230601, P. R. China}
}

\date{}
\maketitle

\begin{abstract}
The spread of a graph is the difference between the largest and smallest
eigenvalue of its adjacency matrix. In this paper, we investigate spread problems for 
graphs with excluded clique-minors. We show that for sufficiently large $n$, the $n$-vertex 
$K_r$-minor free graph with maximum spread is the join of a clique and an independent set, 
with $r-2$ and $n-r+2$ vertices, respectively.
\end{abstract}

\section{Introduction}

Consider an $n \times n$ complex matrix $M$ with eigenvalues $\lambda_1,\lambda_2,\ldots,\lambda_n$. 
The \emph{spread} $s(M)$ of $M$ is defined as $s(M) = \max_{i,j}{|\lambda_i - \lambda_j|}$,
which reflects the largest distance between any two eigenvalues of the matrix. The concept 
of spread was first introduced by Mirsky \cite{Mirsky1956} in 1956, with further significant 
results appearing in \cite{ML}. Since then, the concept of spread has attracted significant 
attention from researchers; see, for example, \cite{Deutsch1978,Johnson1985,NT,TRC}.

The spread of a matrix has also drawn interest in specific cases. Let $G$ be a simple 
undirected graph of order $n$. The adjacency matrix of $G$, denoted by $A(G)$, is 
an $n\times n$ matrix whose rows and columns are indexed by the vertices of $G$. 
The $(u, v)$-entry of $A(G)$ is $1$ if $u$ and $v$ are adjacent, and $0$ otherwise.
Since $A(G)$ is symmetric and real, its eigenvalues are real and can be ordered 
as $\lambda_1(G)\geq\lambda_2(G)\geq\cdots\geq\lambda_n(G)$. In the context of 
the adjacency matrix $A(G)$, the spread is simply the difference between the 
largest and smallest eigenvalues, denoted by
\[
s(G) := \lambda_1(G) - \lambda_n(G).
\] 

Compared with the widely studied largest and smallest eigenvalues, the spread of graphs contains more information about the distribution of eigenvalues and reflects more of the global structure of graphs.
The systematic study of the spread of graphs was initiated by Gregory, Hershkowitz, 
and Kirkland \cite{GHK}. Since then, the spread of graphs has been widely investigated. 
A major focus in this area is to maximize or minimize the spread over a fixed family 
of graphs and characterize the corresponding extremal graphs that achieve these bounds. 
In 2001, Gregory, Hershkowitz, and Kirkland \cite{GHK} made two significant conjectures. 
They first conjectured that the $n$-vertex graph with maximum spread is given by 
$K_{\lfloor2n/3\rfloor} \vee \lceil n/3\rceil K_1$, the join of the clique on 
$\lfloor2n/3\rfloor$ vertices and an independent set on $\lceil n/3\rceil$ vertices.
They also conjectured that if $G$ maximizes spread over all $n$-vertex graphs with $m\leq \lfloor n^2/4 \rfloor$ edges, 
then $G$ must be bipartite. Subsequently, many scholars contributed to these two conjectures, 
such as unicyclic graphs \cite{FWG, LZZ,  WS}, $\infty$-bicyclic graphs 
(the cycles containing two edges that do not intersect) \cite{Z}, bicyclic graphs \cite{WZS}, 
the family of all $n$-vertex graphs \cite{BRTJ}, graphs with a given matching number \cite{LZZ}, 
girth \cite{FWG, WZS}, or size \cite{LL}, very recently for the family of planar and outerplanar graphs 
\cite{GOBT, LLLW, LLW}. We would like to mention that Breen, Riasanovsky, Tait, and 
Urschel \cite{BRTJ} confirmed these two conjectures.

A minor of a graph $G$ is a graph obtained from $G$ by means of a sequence of vertex 
deletions, edge deletions, and edge contractions.
A graph is said to be $H$-minor free, if it does not contain $H$ as a minor.
Minors play a key role in graph theory, and extremal problems
on forbidding minors have attracted appreciable amount of
interest in the past decades, particularly in the case of $K_r$-minors. 
A natural way to construct a large $K_r$-minor free graph is by taking the join of $K_{r-2}$ with $n - r + 2$ independent vertices. This yields a $K_r$-minor free graph with $(r-2)n - \binom{r-1}{2}$ edges. Mader \cite{M} proved that this construction achieves the maximum possible size for $r\leq 7$.
However, this expression no longer holds for $r\geq 8$.

From spectral perspective, spectral extremal problems in graph theory, specifically for graphs excluding certain minors, have been extensively studied in recent decades. 
These investigations cover families such as outerplanar graphs \cite{LN, TT}, planar graphs 
\cite{BR, CV, DM, EZ,  TT}, and $K_r$-minor free graphs \cite{T}.
Notably, Tait \cite{T} investigated the spectral radius of $K_{s,t}$-minor free graphs and posed a conjecture regarding the structure of graphs with the maximum spectral radius within this class. This conjecture was later resolved by Zhai and Lin \cite{ZL}. Recently, Zhai, Fang, and Lin \cite{Zhai-Fang-Lin2024} study some unified phenomena in graphs with maximum spectral radius and excluded general minors.

The main goal of this paper is to contribute to the study of the spread in graphs with excluded minors.

\begin{problem}
Let $H$ be a graph. What is the maximum spread of an $H$-minor free graph of order $n$?
\end{problem}

In recent years, there have been notable developments on this problem.  
Gotshall, O'Brien, and Tait \cite{GOBT} showed that 
for $n$ sufficiently large, the $n$-vertex outerplanar graph with maximum spread is a vertex 
joined to a linear forest. They further conjectured that the extremal 
graph is $K_1\vee P_{n-1}$. Recently, Li, Linz, Lu, and Wang \cite{LLLW} disproved this conjecture by showing that the extremal graph is instead 
$K_1 \vee ( P_{\lceil (2n-1)/3 \rceil} \cup \lfloor (n-2)/3 \rfloor K_1)$. They also proved that the unique planar graph attaining the maximum spread is given by 
$(K_1\cup K_1)\vee(P_{\lceil(2n-2)/3\rceil}\cup\lfloor(n-4)/3\rfloor K_1)$. Furthermore, 
Linz, Lu, and Wang \cite{LLW} determined the maximum spread over all $K_{2,t}$-minor 
free graphs on $n$ vertices for $t \geq 2$ and $n$ sufficiently large. Very recently, 
they extended this result to $K_{s,t}$-minor free graphs for $t \geq s\geq2$ \cite{LLW2}.

In this paper, inspired by the results in \cite{BGWLL, GOBT, LLLW, LLW, LLW2},
we determine the unique graph attaining the maximum spread among $n$-vertex $K_{r}$-minor 
free graphs. Our main result is stated as follows.

\begin{theorem}\label{thm:spread-kr-minor-free}
Let $G$ be a $K_r$-minor free graph of order $n$. For $r\geq3$ and $n$ sufficiently large, we have
\[
s(G)\leq\sqrt{4(r-2)(n-r+2)+(r-3)^2},
\]
with equality if and only if $G\cong K_{r-2}\vee(n-r+2)K_1$. 
\end{theorem}

The rest of the paper is organized as follows. In Section \ref{sec2}, some necessary notations 
and lemmas are provided. In Section \ref{sec3}, we prove that the $K_{r}$-minor free graphs $G$ with maximum spread must contain $K_{r-2, n-r+2}$ as a spanning subgraph for $n$ sufficiently large and $r\geq 3$. 
In Section \ref{sec4}, using 
the novel Laurent series expansion method developed by Li, Linz, Lu, and Wang \cite{LLLW}, 
we show that the maximum spread over all $n$-vertex $K_r$-minor free graphs is achieved 
by $K_{r-2}\vee(n-r+2)K_1$, thereby completing the proof of Theorem \ref{thm:spread-kr-minor-free}.

\section{Preliminaries}
\label{sec2}

In this section we introduce definitions and notation that will be used throughout the paper, 
and prove some preliminary lemmas.

\subsection{Notation}
Consider a simple graph $G$ of order $n$, the number of edges of $G$ is called its size, 
and denoted by $e(G)$. Given a subset $X$ of the vertex set $V(G)$ of $G$, the
subgraph of $G$ induced by $X$ is denoted by $G[X]$, and the graph obtained from $G$ by
deleting $X$ is denoted by $G\setminus X$. As usual, for a vertex $v$ of $G$ we write $d_G(v)$ 
and $N_G(v)$ for the degree of $v$ and the set of neighbors of $v$ in $G$, respectively. 
If the underlying graph $G$ is clear from the context, we use the notations $d(v)$ and $N(v)$. 
Let $N_X(v)$ denote the set of vertices in $X$ that are adjacent to $v$, i.e., $N_X(v) = N_G(v) \cap X$.
Let $G \cup H$ denote the disjoint union of $G$ and $H$. The join $G \vee H$ of disjoint 
graphs of $G$ and $H$ is the graph obtained from $G \cup H$ by joining each vertex of $G$ 
to each vertex of $H$. 

Let $I_n$ denote the identity matrix of order $n$, $\bm{1}_n$ denote the all-ones vector of 
length $n$, and $J_{p\times q}$ denote the all-ones matrix of dimensions $p\times q$. In the 
above notation, we will skip the subscripts when they are clear from context.
For a positive integer $n$, let $[n]$ denote the set $\{1, 2,\ldots, n\}$.


Let $\bm{y}$ be an eigenvector of $G$ corresponding to an eigenvalue $\lambda$ of $G$. 
For each $v \in V(G)$, we have
\[
\lambda y_v = \sum_{u\in N(v)} y_u,
\]
and refer it as the \emph{eigen-equation} with respect to $v$ and $\lambda$.

\subsection{Useful results}
This subsection collects some necessary results required for this paper.

In 1967, Mader \cite{M2} showed a result on the number of edges in $H$-minor free graphs.

\begin{theorem}[\cite{M2}]\label{thm:edge-kr-minor-free}
Let $G$ be an $n$-vertex graph. For every graph $H$, if $G$ is $H$-minor free, then there exists
a constant $C$ such that
\[
e(G) \leq Cn.
\]
\end{theorem}

In 2007, Thomason \cite{TA3} provided an upper bound on the size of $K_{s,t}$-minor
free bipartite graphs.

\begin{theorem}[\cite{TA3}]\label{thm:edge-kst-minor-free-bipartite}
Let $G$ be an $n$-vertex $K_{s,t}$-minor free bipartite graph with two vertex partitions $A$ 
and $B$, where $|A|\gg |B|>0$. Then
\[ 
e(G)\leq (s-1)\cdot|A| + 4^{s+1} s!t \cdot |B|.
\]
\end{theorem}

In 2004, Hong \cite{H} determined the unique graph with the maximum spectral radius among 
$K_5$-minor free graphs. More recently, Tait \cite{T} extended Hong's result to $K_r$-minor free graphs by proving 
the following theorem.

\begin{theorem}[\cite{T}]\label{thm:spectral-kr-minor-free}
Let $G$ be the $n$-vertex $K_r$-minor free graphs with $r\geq3$. If $n$ is large enough, 
then 
\[
\lambda_1(G) \leq\frac{1}{2}\left(r-3+\sqrt{4(r-2)(n-r+2) + (r-3)^2}\right),
\]
with equality holds if and only if $G\cong K_{r-2}\vee (n - r + 2) K_1$.
\end{theorem}

We conclude this subsection with the following lemma, which can be proved by induction or double counting.

\begin{lemma}[{\cite[Lemma 12]{Cioaba-Feng-Tait-Zhang2020}}]\label{lem:intersection-sets}
Let $A_1,A_2,\ldots,A_k$ be $k$ finite sets. Then 
\[
|A_1\cap A_2\cap\cdots\cap A_k| \geq \sum_{i=1}^k |A_i| - (k-1) \bigg|\bigcup_{i=1}^k A_i\bigg|.
\]
\end{lemma}

\subsection{Preliminary results on $K_r$-minor free graphs}

In light of Theorem \ref{thm:edge-kst-minor-free-bipartite}, we have the following result 
on $K_r$-minor free bipartite graphs.

\begin{lemma}\label{lem:edge-kr-minor-free-bipartite}
Let $G$ be an $n$-vertex $K_r$-minor free bipartite graph, with vertex partition $A$ and $B$.
If $|A| = k$ and $|B| = n - k$, then there exists a constant $C$ depending only on $r$ such that
\[
e(G)\leq (r-2)n + Ck.
\]
\end{lemma}

\begin{proof}
Noting that every $K_{r-1,r-1}$-minor graph contains a $K_r$-minor by contracting some edges, 
we have that $G$ is $K_{r-1,r-1}$-minor free. Therefore, the assertion follows from Theorem
\ref{thm:edge-kst-minor-free-bipartite}, as desired.
\end{proof}

Throughout this paper, we set
\[
\gamma_{n,r} := \sqrt{(r-2)(n-r+2)}.
\]
Using Theorem \ref{thm:spectral-kr-minor-free}, we can now prove the following result.

\begin{lemma}\label{lem:bounds-lambda-1-n}
Let $G$ be an $n$-vertex $K_r$-minor free graph with maximum spread. Then
\[
\gamma_{n,r} - \frac{r-3}{2} - O\Big(\frac{1}{\sqrt{n}}\Big) 
\leq -\lambda_n (G) \leq \lambda_1 (G) \leq\gamma_{n,r} + \frac{r-3}{2} + O\Big(\frac{1}{\sqrt{n}}\Big).
\]
\end{lemma}

\begin{proof}  
By Theorem \ref{thm:spectral-kr-minor-free}, we have
\begin{align}\label{lem:bounds-lambda-1-n-eq:1}
\lambda_1 (G)
& \leq\frac{1}{2}\left(r-3 + \sqrt{4\gamma_{n,r}^2+(r-3)^2}\right) \nonumber \\
& = \gamma_{n,r} + \frac{r-3}{2} + O\Big(\frac{1}{\sqrt{n}}\Big).
\end{align}

Next, we establish the lower bound for $-\lambda_n(G)$. Observe that $K_{r-2,n-r+2}$ is $K_r$-minor 
free, which implies $s(G)\geq s(K_{r-2,n-r+2}) = 2\gamma_{n,r}$. Thus, we have
\begin{equation}\label{lem:bounds-lambda-1-n-eq:2}
-\lambda_n (G) \geq 2\gamma_{n,r} - \lambda_1 (G).
\end{equation}
Substituting \eqref{lem:bounds-lambda-1-n-eq:1} into \eqref{lem:bounds-lambda-1-n-eq:2}, we obtain
\[
-\lambda_n (G) \geq \gamma_{n,r} - \frac{r-3}{2} - O\Big(\frac{1}{\sqrt{n}}\Big).
\]
This completes the proof of Lemma \ref{lem:bounds-lambda-1-n}.
\end{proof}

\section{Structure of $K_r$-minor free graphs with maximum spread}
\label{sec3}

In this section,  we always assume that $G$ is a graph attaining the maximum spread among all 
$K_r$-minor free graphs of order $n$. The aim of this section is to explore the general 
structure of $G$. We will show that $G$ contains $r-2$ vertices, each with degree $n-1$, 
and that after removing these $r-2$ vertices, the resulant graph forms an independent set.

We now fix some notation. Let $\bm{x}$ be a nonnegative eigenvector corresponding to 
$\lambda_1(G)$, and let $\bm{z}$ be an eigenvector corresponding to $\lambda_n(G)$. 
For convenience, we normalize both $\bm{x}$ and $\bm{z}$ so that their maximum entries 
in absolute value are $1$. Without loss of generality, assume that there are two vertices 
$u_0$ and $w_0$ such that $x_{u_0}=1$ and $|z_{w_0}|=1$. Set $\lambda_1 := \lambda_1(G)$ 
and $\lambda_n := \lambda_n (G)$ for short. Let 
\[ 
V_+ = \{v: z_v > 0\},~~ V_- = \{v: z_v < 0\},~~ V_0 = \{v: z_v = 0\}.
\]
Observing that $\bm{z}$ is a non-zero vector and the eigen-equations with respect to $\lambda_n$, 
we derive $V_+\neq\emptyset$ and $V_-\neq\emptyset$.
In this section, we always assume that $|V_+|\leq n/2$.
For $S\subseteq V(G)$, we shall use $\vol(S) = \sum_{v\in S} |z_v|$ to denote the volume of $S$. 

Our sketchy strategy will be to firstly show that there are $r-2$ vertices of large degree 
(Lemma \ref{lem:large-degree}) and the remaining vertices have small eigenvector entry (Lemma \ref{lem:size-xi-in-L-and-U-V}).
Then, we use these facts to show that the $r-2$ vertices of large degree must be adjacent 
to all other vertices (see Theorem \ref{thm:structure-kr-minor-free-max-spread}).

Next, we will present a few simple lemmas that, while perhaps not optimal, are sufficient for our purposes. To begin with, we firstly give the following two facts for estimating the volume of a subset $S$ of $V_+$, which will be used in following lemmas frequently. By considering the eigen-equation with repect to each $v \in S$, we have 
\begin{align} \label{Vol-for-set}
|\lambda_n|^2\vol (S)
& = |\lambda_n|^2\sum_{v \in S} z_v 
= |\lambda_n|\sum_{v \in S} \bigg(-\sum_{u\in N(v)}{z_u}\bigg) \nonumber \\
& \leq\sum_{v \in S} \sum_{u\in N(v)\cap V_-} |\lambda_n| |z_u| \nonumber\\
& = \sum_{v \in S} \sum_{u\in N(v)\cap V_-} \sum_{w\in N(u)} z_w \nonumber\\
& \leq\sum_{v \in S} \sum_{u\in N(v)\cap V_-} \sum_{w\in N(u)\cap V_+} z_w \nonumber\\
& = \sum_{w\in V_+} z_w \sum_{u\in N(w)\cap V_-} |N(u)\cap S| \nonumber\\
& = \sum_{w\in V_+} z_w\cdot |E(N(w)\cap V_-, S)|.
\end{align}
In particular, if $S=\{s\} \subseteq V_+$, it follows from \eqref{Vol-for-set} that
\begin{align}\label{Vol-for-vertex}
   |\lambda_n|^2 z_s  \leq \sum_{w\in V_+} z_w\cdot |E(N(w)\cap V_-, S)|=\sum_{w\in V_+} z_w\cdot |N(w)\cap V_- \cap N(s)|.
\end{align}

\begin{lemma}\label{lem:vol-v+}
$\vol (V_+) = O(1)$.
\end{lemma}

\begin{proof}
Fix a sufficiently small constant $\varepsilon > 0$, let
\begin{equation*}
V_+^{(1)} = \{v\in V_+: |N(v)\cap V_-| \geq\varepsilon n \}, \quad 
V_+^{(2)} = V_+ \setminus V_+^{(1)}.
\end{equation*}
By the definition of $V_+^{(1)}$ and Theorem \ref{thm:edge-kr-minor-free}, we have
\begin{equation}\label{lem:vol-v+-eq:1}
|V_+^{(1)}|\leq\frac{E(V_+^{(1)}, V_-)}{\varepsilon n} = O (1).
\end{equation}

To finish the proof, we shall derive an relationship between $\vol (V_+^{(1)})$ 
and $\vol (V_+^{(2)})$. It follows from \eqref{Vol-for-set} that 
\begin{equation}\label{lem:vol-v+-eq:2}
|\lambda_n|^2 \vol (V_+^{(2)}) \leq 
\sum_{w\in V_+^{(1)}} z_w \cdot |E(N(w)\cap V_-, V_+^{(2)})| + \sum_{w\in V_+^{(2)}} z_w\cdot |E(N(w)\cap V_-, V_+^{(2)})|.
\end{equation}
For each $w\in V_+^{(1)}$, by Theorem \ref{thm:edge-kr-minor-free}, we have
\begin{align}\label{lem:vol-v+-eq:3}
|E(N(w)\cap V_-, V_+^{(2)})|\leq|E(V_-, V_+^{(2)})| = O(n).
\end{align}
For each $w\in V_+^{(2)}$, it follows from Lemma \ref{lem:edge-kr-minor-free-bipartite} that
\begin{align}\label{lem:vol-v+-eq:4}
|E(N(w)\cap V_-, V_+^{(2)})| \leq (r-2) |V_+^{(2)}| + \varepsilon \cdot O(n)
\leq \frac{r-2}{2} n + \varepsilon \cdot O(n).
\end{align}
Hence, by \eqref{lem:vol-v+-eq:2}, \eqref{lem:vol-v+-eq:3} and \eqref{lem:vol-v+-eq:4}, we derive that
\[
|\lambda_n|^2\vol (V_+^{(2)}) \leq \vol(V_+^{(1)}) \cdot O(n) + 
\vol(V_+^{(2)}) \left(\frac{r-2}{2} n + \varepsilon \cdot O(n)\right),
\]
which, together with Lemma \ref{lem:bounds-lambda-1-n}, implies that
\[
\vol(V_+^{(2)}) \leq \frac{O(n)}{|\lambda_n|^2 - (r-2)n/2 - \varepsilon \cdot O(n)} 
\cdot \vol(V_+^{(1)}) = O(1) \cdot \vol(V_+^{(1)}).
\]
Noting that $\vol(V_+) = \vol(V_+^{(1)}) + \vol(V_+^{(2)})$ and \eqref{lem:vol-v+-eq:1}, we derive
\[
\vol(V_+) \leq O(1)\cdot \vol(V_+^{(1)}) \leq O(1) \cdot |V_+^{(1)}| = O(1).
\]
This completes the proof of Lemma \ref{lem:vol-v+}.
\end{proof}

\begin{corollary}\label{cor:size-zi-and-neighbor-w0}
The following conclusions hold:
\begin{enumerate}
\item[\textup{(a)}] For each $v\in V_-$, $|z_v| = O (\frac{1}{\sqrt{n}})$.

\item[\textup{(b)}] $w_0\in V_+$.

\item[\textup{(c)}] $|N(w_0)\cap V_-|=\Omega(n)$.
\end{enumerate}
\end{corollary}

\begin{proof}
(a) For any $v\in V_-$, by the eigen-equation with respect to $v$, we have
\[
|\lambda_n||z_v| = \sum_{u\in N(v)} z_u
\leq\sum_{u\in N(v)\cap V_+} z_u\leq\vol(V_+).
\]
Combining this with Lemma \ref{lem:bounds-lambda-1-n} and Lemma \ref{lem:vol-v+}, we obtain
\[
|z_v|\leq\frac{\vol(V_+)}{|\lambda_n|} = O\Big(\frac{1}{\sqrt{n}}\Big).
\]

(b) Since $|z_{w_0}|=1$ and $|z_v| = O (\frac{1}{\sqrt{n}})$ for each $v\in V_-$, it follows 
that $w_0\in V_+$ and $z_{w_0}=1$.

(c) By \eqref{Vol-for-vertex}, we have
\begin{align*}
|\lambda_n|^2 =|\lambda_n|^2 z_{w_0}
& \leq \sum_{v\in V_+} z_v \cdot|N(v)\cap V_-\cap N(w_0)|  \\
& \leq |N(w_0)\cap V_-|\cdot\vol(V_+).
\end{align*}
This implies $|N(w_0)\cap V_-|\geq\frac{|\lambda_n|^2}{\vol(V_+)}$.
Then, by Lemma \ref{lem:bounds-lambda-1-n} and Lemma \ref{lem:vol-v+}, we have
\[
|N(w_0)\cap V_-| = \Omega(n).
\]
This completes the proof of Corollary \ref{cor:size-zi-and-neighbor-w0}.
\end{proof}

\begin{lemma}\label{lem:size-V+-and-V-minus}
The following conclusions hold:
\begin{enumerate}
\item[\textup{(a)}] $|V_-|\geq n - O(\sqrt{n})$.

\item[\textup{(b)}] $|V_+| = O(\sqrt{n})$.
\end{enumerate}
\end{lemma}

\begin{proof}
Choose a fixed constant $C_1$, let
\begin{align}\label{lem:size-V+-and-V-minus-eq:1}
V_+^{(3)} = \{v\in V_+: |N(v)\cap V_-|\geq C_1\sqrt{n}\}, \quad 
V_+^{(4)} = V_+\setminus V_+^{(3)}.
\end{align}
By the definition of $V_+$ and Theorem \ref{thm:edge-kr-minor-free}, we obtain
\begin{equation}\label{lem:size-V+-and-V-minus-eq:2}
|V_+^{(3)}| \leq \frac{E(V_+^{(3)}, V_-)}{C_1\sqrt{n}} \leq O(\sqrt{n}).
\end{equation}
By \eqref{Vol-for-set}, we have
\begin{align}\label{lem:size-V+-and-V-minus-eq:3}
|\lambda_n|^2 \vol(V_+^{(3)})
 \leq  \sum_{w\in V_+^{(3)}} z_w\cdot|E(N(w)\cap V_-, V_+^{(3)})| 
+ \sum_{w\in V_+^{(4)}} z_w\cdot|E(N(w)\cap V_-, V_+^{(3)})|.
\end{align}
For each $w\in V_+^{(3)}$, it follows from Lemma \ref{lem:edge-kr-minor-free-bipartite} that
\begin{equation}\label{lem:size-V+-and-V-minus-eq:4}
|E(N(w)\cap V_-, V_+^{(3)})| \leq |E(V_-, V_+^{(3)})| \leq (r-2) |V_-| + O(1) \cdot |V_+^{(3)}|.
\end{equation}
For each $w\in V_+^{(4)}$, by Lemma \ref{lem:edge-kr-minor-free-bipartite} and the definition of $V_+^{(4)}$, we obtain
\begin{equation}\label{lem:size-V+-and-V-minus-eq:5}
|E(N(w)\cap V_-, V_+^{(3)})| \leq (r-2) |V_+^{(3)}| + O(\sqrt{n}).
\end{equation}
Then, by \eqref{lem:size-V+-and-V-minus-eq:3}, \eqref{lem:size-V+-and-V-minus-eq:4} 
and \eqref{lem:size-V+-and-V-minus-eq:5}, we can deduce that
\[
|\lambda_n|^2 \vol(V_+^{(3)}) \leq \vol(V_+^{(3)}) \big[(r-2) |V_-| + O(1) \cdot |V_+^{(3)}|\big] 
+ \vol(V_+^{(4)}) \big[(r-2) |V_+^{(3)}| + O(\sqrt{n})\big],
\]
which implies that
\begin{equation}\label{lem:size-V+-and-V-minus-eq:6}
|V_-| \geq \frac{|\lambda_n|^2}{r-2} - O(1) \cdot |V_+^{(3)}| -\big(|V_+^{(3)}| + O(\sqrt{n})\big) \frac{\vol(V_+^{(4)})}{\vol(V_+^{(3)})}.
\end{equation}
Note that, by \eqref{lem:size-V+-and-V-minus-eq:1} and Corollary \ref{cor:size-zi-and-neighbor-w0} (c), 
we have  $w_0\in V_+^{(3)}$, and so,
\begin{align}\label{lem:size-V+-and-V-minus-eq:7}
\vol(V_+^{(3)})\geq 1, \quad \vol(V_+^{(3)}) \leq \vol (V_+) = O(1).
\end{align}
Then, by \eqref{lem:size-V+-and-V-minus-eq:1}, \eqref{lem:size-V+-and-V-minus-eq:6}, \eqref{lem:size-V+-and-V-minus-eq:7}, 
and Lemma \ref{lem:bounds-lambda-1-n}, we have $|V_-|\geq n - O(\sqrt{n})$, and so, $|V_+| = O(\sqrt{n})$.

This completes the proof of Lemma \ref{lem:size-V+-and-V-minus}.
\end{proof}

In what follows, let
\begin{equation}\label{eq:new-V+-and-V-minus}
V_+^{'} = \{v\in V_+: |N(v)\cap V_-|\geq n - C_2\sqrt{n}\}, \quad V_+^{''} = V_+\setminus V_+^{'},
\end{equation}
where $C_2$ is some large constant and chosen in next Lemma.

\begin{lemma}\label{lem:vol-V+''-and-V+}
We have
\begin{align*}
\vol(V_+^{''}) & \leq \left(\frac{1}{r-1} + O\Big(\frac{1}{\sqrt{n}}\Big)\right) \vol(V_+^{'}), \\
\vol(V_+) & \leq \left(\frac{r}{r-1} + O\Big(\frac{1}{\sqrt{n}}\Big)\right) \vol(V_+^{'}).
\end{align*}
\end{lemma}

\begin{proof}
It follows from \eqref{Vol-for-set} that
\begin{align}\label{lem:vol-V+''-and-V+-eq:1}
|\lambda_n|^2\vol(V_+)
\leq  \sum_{w\in V_+'} z_w \cdot |E(N(w)\cap V_-,V_+)| 
+ \sum_{w\in V_+^{''}} z_w\cdot|E(N(w)\cap V_-,V_+)|.
\end{align}
For each $w\in V_+'$, it follows from Lemma \ref{lem:edge-kr-minor-free-bipartite} and 
Lemma \ref{lem:size-V+-and-V-minus} that
\begin{align}\label{lem:vol-V+''-and-V+-eq:2}
|E(N(w)\cap V_-, V_+)|
& \leq |E(V_-,V_+)| \leq (r-2)|V_-| + O(1) \cdot |V_+| \nonumber \\
& \leq (r-2)n + O(\sqrt{n}).
\end{align}
For each $w\in V_+^{''}$, by Lemma \ref{lem:edge-kr-minor-free-bipartite}, we have
\begin{align}\label{lem:vol-V+''-and-V+-eq:3}
|E(N(w)\cap V_-, V_+)|
& \leq (r-2)|N(w)\cap V_-| + O(1) \cdot |V_+| \nonumber \\
& \leq (r-2) (n - C_2\sqrt{n}) + O(\sqrt{n}).
\end{align}
Then, by \eqref{lem:vol-V+''-and-V+-eq:1}, \eqref{lem:vol-V+''-and-V+-eq:2} 
and \eqref{lem:vol-V+''-and-V+-eq:3}, we find that
\begin{align*}
|\lambda_n|^2 \vol(V_+) 
& \leq \vol(V_+') \left[ (r-2) n + O(\sqrt{n})\right] + \vol(V_+^{''}) \left[(r-2)(n-C_2\sqrt{n}) + O(\sqrt{n})\right] \\
& = \vol(V_+) \left[(r-2) n + O(\sqrt{n})\right] - \vol(V_+^{''})(r-2)C_2 \sqrt{n},
\end{align*}
which implies that
\begin{align}\label{lem:vol-V+''-and-V+-eq:4}
\vol(V_+^{''}) \leq \frac{(r-2)n + O(\sqrt{n}) - |\lambda_n|^2}{(r-2) C_2\sqrt{n}} \vol(V_+).
\end{align}
By Lemma \ref{lem:bounds-lambda-1-n}, we can choose some constant $C_2$ such that 
\begin{align}\label{lem:vol-V+''-and-V+-eq:5}
\frac{(r-2)n + O(\sqrt{n}) - |\lambda_n|^2}{(r-2)C_2\sqrt{n}} \leq \frac{1}{r} + O\Big(\frac{1}{\sqrt{n}}\Big).
\end{align}
Then, it follows from \eqref{lem:vol-V+''-and-V+-eq:4} and \eqref{lem:vol-V+''-and-V+-eq:5} that
\begin{align*}
\vol(V_+^{''}) \leq \left(\frac{1}{r} + O\Big(\frac{1}{\sqrt{n}}\Big)\right) \vol(V_+).
\end{align*}
Recall that $\vol(V_+)=\vol(V_+')+\vol(V_+^{''})$. From this, we can derive
\begin{align*}
\vol(V_+^{''})
& \leq \frac{\frac{1}{r} + O(\frac{1}{\sqrt{n}})}{1 - \frac{1}{r} - O(\frac{1}{\sqrt{n}})} \cdot\vol(V_+^{'}) \\
& = \left(\frac{1}{r-1} + O\Big(\frac{1}{\sqrt{n}}\Big)\right) \vol(V_+').
\end{align*}
We can also get that
\[
\vol(V_+) \leq \left(\frac{r}{r-1} + O\Big(\frac{1}{\sqrt{n}}\Big)\right) \vol(V_+^{'}).
\]

This completes the proof of Lemma \ref{lem:vol-V+''-and-V+}.
\end{proof}

Now, we shall study the eigen-components of the eigenvector $\bm{z}$ corresponding to $\lambda_n$.
\begin{lemma}\label{lem:large-degree}
The following conclusions hold:
\begin{enumerate}
\item[\textup{(a)}] $|V_+^{'}| = r-2$, and $w_0\in V_+^{'}$.

\item[\textup{(b)}] For each $v\in V(G)\setminus V_+^{'}$, $d_{V_-}(v)=O(\sqrt{n})$.

\item[\textup{(c)}] For each $v\notin V_+^{'}$, $|z_v| = O(n^{-1/2})$.
 
\item[\textup{(d)}] For each $v\in V_+'$, $z_v\geq 1 - O (n^{-1/2})$.

\item[\textup{(e)}] For each $v\in V_+^{'}$, $d(v) \geq n - O(\sqrt{n})$.
\end{enumerate}    
\end{lemma}

\begin{proof}
(a) As the first step, we claim that $|V_+^{'}|\leq r-2$.
Assume to the contrary, suppose that $|V_+^{'}|\geq r-1$. Without loss of generality, we may assume that 
$\{v_1,v_2,\ldots,v_{r-1}\}\subseteq V_+^{'}$. Then, for large enough $n$, by 
Lemma \ref{lem:intersection-sets} and \eqref{eq:new-V+-and-V-minus},
\begin{align*}
\left|\bigcap_{i=1}^{r-1} N_{V_-}(v_i)\right|
& \geq \sum_{i=1}^{r-1} |N_{V_-}(v_i)| - (r-2) \left|\bigcup_{i=1}^{r-1} N_{V_-}(v_i)\right| \\
& \geq (r-1) (n - C_2\sqrt{n}) - (r-2) |V_-| \\
& \geq (r-1) (|V_-| - C_2\sqrt{n}) - (r-2) |V_-| \\
& = |V_-| - (r-1) C_2\sqrt{n} \\
& > r-1.
\end{align*}
    
Let $u_1,u_2,\ldots,u_{r-1}$ be $r-1$ vertices in $\bigcap_{i=1}^{r-1} N_{V_-}(v_i)$.
By contracting the edges $v_1u_1$, $v_2u_2$,$\ldots$,$v_{r-2}u_{r-2}$ and combining the vertices 
$\{v_{r-1},u_{r-1}\}$, we obtain a $K_r$ minor of $G$, a contradiction. Hence, we have
\[
\vol(V_+^{'})\leq|V_+^{'}|\leq r-2,
\]
and by Lemma \ref{lem:vol-V+''-and-V+}, it is easy to see that
\[
\vol(V_+^{''})\leq \left(\frac{1}{r-1} + O\Big(\frac{1}{\sqrt{n}}\Big)\right)\cdot (r-2) < 1.
\]
This implies that $w_0\notin V_+^{''}$ since $z_{w_0} = 1$.
    
Next, we shall show that $|V_+^{'}| \geq r-2$. By contradiction, suppose that 
$|V_+^{'}|\leq r-3$. For $w_0\in V_+^{'}$, by \eqref{Vol-for-vertex}, we have
\begin{align*}
|\lambda_n|^2 = |\lambda_n|^2 z_{w_0} 
& \leq \sum_{w\in V_+} z_w \cdot |N(w)\cap V_-\cap N(w_0)| \\
& \leq |V_+^{'}| \cdot d_{V_-}(w_0) + \sum_{w\in V_+^{''}} d_{V_-}(w_0) z_w \\
& \leq |V_+^{'} |\cdot n + \vol(V_+^{''}) \cdot n. 
\end{align*}
It follows from Lemma \ref{lem:vol-V+''-and-V+} and $|V_+^{'}|\leq r-3$ that
\begin{align*}
|\lambda_n|^2
& \leq |V_+^{'}| \cdot n + \left(\frac{1}{r-1} + O\Big(\frac{1}{\sqrt{n}}\Big)\right) \vol(V_+^{'}) \cdot n \\
& \leq (r-3) n + \left(\frac{1}{r-1} + O\Big(\frac{1}{\sqrt{n}}\Big)\right)\cdot (r-3)\cdot n \\
& = \left(r-2 - \frac{2}{r-1}\right)n + O(\sqrt{n}),
\end{align*}
contradicting the lower bound of $|\lambda_n|$. Hence, $|V_+'|=r-2$.
    
(b) For each $v \notin V_+^{'}$, we shall prove $d_{V_-}(v) \leq (r-1)C_2\sqrt{n}$. 
Assume to the contrary, suppose that $d_{V_-}(v) > (r-1)C_2\sqrt{n}$.
Let $V_+^{'} = \{v_1,v_2,\ldots,v_{r-2}\}$. Then for large enough $n$, by 
Lemma \ref{lem:intersection-sets} and \eqref{eq:new-V+-and-V-minus},
\begin{align*}
\left|N_{V_-}(v)\cap \bigg(\bigcap_{i=1}^{r-2} N_{V_-}(v_i)\bigg)\right|
& \geq d_{V_-}(v) + \sum_{i=1}^{r-2} |N_{V_-}(v_i)| - (r-2)\left|N_{V_-}(v)\cup\left(\bigcup_{i=1}^{r-2}{N_{V_-}(v_i)}\right)\right| \\
& \geq d_{V_-}(v) + \sum_{i=1}^{r-2} |N_{V_-}(v_i)| -(r-2) |V_-| \\
& > (r-1)C_2\sqrt{n}+(r-2)(n-C_2\sqrt{n}) - (r-2)|V_-| \\
& > (r-1)C_2\sqrt{n} - (r-2)C_2\sqrt{n} \\
& > r-1.
\end{align*}
    
Let $u_1,u_2,\ldots,u_{r-1}$ be $r-1$ vertices in the common neighbors. By contracting the edges $v_1u_1, v_2u_2, \ldots, v_{r-2}u_{r-2}$ 
and combining the vertices $\{v,u_{r-1}\}$, we obtain a $K_r$ minor of $G$, which contradicts the fact that $G$ is 
$K_r$-minor free.
    
(c) Recall Corollary \ref{cor:size-zi-and-neighbor-w0} (a), we have $|z_v| = O(n^{-1/2})$ for each $v\in V_-$.
Thus, it is sufficient to consider the case $v\in V_+^{''}$.
By \eqref{Vol-for-vertex}, we have
\begin{align}\label{lem:large-degree-eq:1}
|\lambda_n|^2 z_v
& = \sum_{w\in V_+} z_w \cdot |N(w)\cap V_- \cap N(v)| \nonumber \\
& \leq \sum_{w\in V_+} z_w d_{V_-}(v) \nonumber \\
& = O(\sqrt{n}),
\end{align}
where the last equality follows from Lemma \ref{lem:vol-v+} and Lemma \ref{lem:large-degree} (b).
By Lemma \ref{lem:bounds-lambda-1-n} and (\ref{lem:large-degree-eq:1}), we have
\[
z_v \leq \frac{O(\sqrt{n})}{|\lambda_n|^2}=O\left(\frac{1}{\sqrt{n}}\right).
\]
    
(d) For each $v\in V_+'\setminus\{w_0\}$, we can deduce that
\begin{align}\label{lem:large-degree-eq:2}
|\lambda_n| (1 - z_v)
& = -\sum_{u\in N(w_0)} z_u + \sum_{w\in N(v)} z_w \nonumber \\
& = -\sum_{u\in N(w_0)\setminus N(v)} z_u 
+ \sum_{w\in N(v)\setminus N(w_0)} z_w \nonumber \\
& \leq \sum_{u\in (N(w_0)\setminus N(v))\cap V_-} |z_u| 
+ \sum_{w\in (N(v)\setminus N(w_0))\cap V_+} z_w \nonumber \\
& \stackrel{\footnotesize{\textcircled{\scriptsize{1}}}\footnotesize}
\leq |N_{V_-}(w_0)\backslash N_{V_-}(v)| \cdot O(\frac{1}{\sqrt{n}}) + \vol(V_+) \nonumber \\
& \stackrel{\footnotesize{\textcircled{\scriptsize{2}}}\footnotesize}
\leq (|V_-| - |N_{V_-}(v)|) \cdot O(\frac{1}{\sqrt{n}}) + O(1) \nonumber \\
& \stackrel{\footnotesize{\textcircled{\scriptsize{3}}}\footnotesize}
\leq \big( n - (n - C_2\sqrt{n})\big) \cdot O(\frac{1}{\sqrt{n}}) + O(1) \nonumber \\
& = O(1),
\end{align}   
where \normalsize{\textcircled{\scriptsize{1}}}\normalsize\enspace follows from Lemma \ref{lem:large-degree} (c), \normalsize{\textcircled{\scriptsize{2}}}\normalsize\enspace follows from Lemma \ref{lem:vol-v+}, 
and \normalsize{\textcircled{\scriptsize{3}}}\normalsize\enspace follows from \eqref{eq:new-V+-and-V-minus}.  
The result immediately follows from Lemma \ref{lem:bounds-lambda-1-n} and \eqref{lem:large-degree-eq:2}. 
    
(e) 
For each vertex $v\in V_+'$,  by \eqref{Vol-for-vertex}, we have 
\begin{align}\label{lem:large-degree-eq:4}
|\lambda_n|^2 z_v
& \leq \sum_{w\in V_+}  z_w \cdot |N(w)\cap V_-\cap N(v)| z_w \nonumber \\
& \leq \sum_{w\in V_+^{'}} d_{V_-}(v) z_w + O(\sqrt{n}) \nonumber \\
& \leq (r-2) d_{V_-}(v) + O(\sqrt{n}),
\end{align}
and so, by Lemma \ref{lem:bounds-lambda-1-n}, Lemma \ref{lem:large-degree} (d), and \eqref{lem:large-degree-eq:4}, we obtain
\begin{align*}
d(v) \geq d_{V_-}(v)
\geq \frac{|\lambda_n|^2}{r-2} \cdot z_v - O(\sqrt{n})
\geq n- O(\sqrt{n}).
\end{align*}
This completes the proof of Lemma \ref{lem:large-degree}.
\end{proof}

The Lemma \ref{lem:large-degree}(a) establishes that there are exactly $r-2$ vertices in $V_+'$ and $w_0 \in V_+'$.
Let us denote $V_+' = \{v_1=w_0,v_2,\ldots,v_{r-2}\}$. For convenience, we set
\begin{align} \label{def-LUV}
L & := V_+', \nonumber\\
U & := \{v\in V(G)\setminus L:~|N(v)\cap L| = r-2\},  \\
V & := V(G)\setminus (L\cup U). \nonumber
\end{align}
In the following consequence, we shall give some rough characterizations for $U$ and $V$.

\begin{lemma}\label{lem:L-U-V}
For sufficiently large $n$, the following statements hold. 
\begin{enumerate}
\item[\textup{(a)}] $|U|\geq n - O(\sqrt{n})$.

\item[\textup{(b)}] For each $v\in U$, $N(v)\cap U = \emptyset$ .

\item[\textup{(c)}] For each $v\in V$, $|N(v)\cap U|\leq 1$ 

\item[\textup{(d)}] For any vertices $p,q\in V$, if each of them have exactly 
one neighbor in $U$, set $N(p)\cap U =:\tilde{p}$, $N(q)\cap U =:\tilde{q}$.
If $\tilde{p}\neq \tilde{q}$, then $p$ and $q$ lie in different components of $G[V]$.
\end{enumerate}      
\end{lemma}

\begin{proof}
(a) Note that $|U| = \big|\bigcap_{i=1}^{r-2} N(v_i)\big|$. Hence, by Lemma \ref{lem:intersection-sets} 
and Lemma \ref{lem:large-degree} (e),  we have
\begin{align*}
|U| & \geq \sum_{i=1}^{r-2} |N(v_i)| - (r-3) \left|\bigcup_{i=1}^{r-2} N(v_i)\right| \\
& \geq (r-2) \big(n - O(\sqrt{n})\big) - (r-3) n \\
& = n - O(\sqrt{n}).
\end{align*}
This proves statement (a).
    
(b) Assume to the contrary, suppose $|N(v)\cap U|\geq 1$ for some $v\in U$, then there is an edge $st\in E(G[U])$.
Since $|U|\geq n - O(\sqrt{n})$, we can select $r-3$ vertices $\{u_1,u_2,\ldots,u_{r-3}\}\subseteq U$.
Recall that $L = \{v_1,v_2,\ldots,v_{r-2}\}$. By contracting the edges $v_1u_1,v_2u_2,\ldots,v_{r-3}u_{r-3}$ 
and combining the vertices $\{v_{r-2},s,t\}$, we see that $G$ has a $K_r$ minor, a contradiction.

(c) Assume to the contrary, suppose that there is a vertex $v\in V$ such that $N(v)\cap U = \{s_1,s_2\}$.
By contracting the path $s_1vs_2$ into the edge $s_1s_2$, we create an edge within $U$. Using the 
same arguments as statement (b), we can identify a $K_r$-minor in $G$, leading to a contradiction.
    
(d) For two arbitrary vertices $p,q\in V$, suppose $p$ and $q$ have exactly one neighbor 
$\tilde{p}$ and $\tilde{q}$ in $U$, respectively, and $\tilde{p}\neq \tilde{q}$.
Assume, for contradiction, that $p$ and $q$ are connected by a path $P$ in $G[V]$.
Then, we can contract the path $\tilde{p}pPq\tilde{q}$ into an edge $\tilde{p}\tilde{q}$. 
Using the same arguments as statement (b), we can identify a $K_r$-minor in $G$, thereby proving statement (d).
\end{proof}

Recall that $\bm{x}$ is the normalized eigenvector corresponding to the spectral radius of $A(G)$, 
and $u_0$ is a vertex such that $x_{u_0}=1$. Next, we shall investigate the components of the eigenvector $\bm{x}$.

\begin{lemma}\label{lem:size-xi-in-L-and-U-V}
For sufficiently large $n$, we have
\begin{enumerate}
\item[\textup{(a)}] $u_0\in L$.

\item[\textup{(b)}] $x_v \geq 1 - O(\frac{1}{\sqrt{n}})$ for $v\in L\setminus\{u_0\}$.

\item[\textup{(c)}] $x_v = O(\frac{1}{\sqrt{n}})$ for $v\notin L$.
\end{enumerate}       
\end{lemma}

\begin{proof}
We now simultaneously establish (a) and (c). For each $v\notin L$, we have
\begin{align}\label{lem:size-xi-in-L-and-U-V-eq:1}
\lambda_1^2x_v
& = \lambda_1 \sum_{u\in N(v)} x_u \nonumber \\
& = \lambda_1 \Bigg(\sum_{u\in N(v)\cap U} x_u + \sum_{u\in N(v)\cap L} x_u
+ \sum_{u\in N(v)\cap V} x_u\Bigg) \nonumber \\
& \stackrel{\footnotesize{\textcircled{\scriptsize{1}}}\footnotesize}
\leq (r-2)\lambda_1 + \sum_{u\in N(v)\cap V}{\sum_{w\in N(u)}{x_w}} \nonumber \\
& = (r-2)\lambda_1 + \sum_{u\in N(v)\cap V} 
\Bigg(\sum_{w\in N(u)\cap U} x_w + \sum_{w\in N(u)\cap L} x_w + \sum_{w\in N(u)\cap V} x_w \Bigg) \nonumber \\
& \stackrel{\footnotesize{\textcircled{\scriptsize{2}}}\footnotesize}
\leq (r-2) \lambda_1 + |N(v)\cap V| + (r-3) |N(v)\cap V| + \sum_{u\in N(v)\cap V} \sum_{w\in N(u)\cap V} x_w \nonumber \\
& \leq (r-2) \lambda_1 + (r-2) |N(v)\cap V| + 2|E(G[V])| \nonumber \\
& \stackrel{\footnotesize{\textcircled{\scriptsize{3}}}\footnotesize} = (r-2)\lambda_1 + O(|V|).
\end{align}
Here, both \normalsize{\textcircled{\scriptsize{1}}}\normalsize\enspace and \normalsize{\textcircled{\scriptsize{2}}}\normalsize\enspace can be derived from Lemma \ref{lem:L-U-V} (b),(c) and \eqref{def-LUV}, and \normalsize{\textcircled{\scriptsize{3}}}\normalsize\enspace follows from Theorem \ref{thm:edge-kr-minor-free}. 
By Lemma \ref{lem:L-U-V} (a), we obtain
\begin{align}\label{lem:size-xi-in-L-and-U-V-eq:2}
|V| = n - |L| - |U| \leq n-(r-2) - (n-O(\sqrt{n})) = O(\sqrt{n}).
\end{align}
It follows from Lemma \ref{lem:bounds-lambda-1-n}, \eqref{lem:size-xi-in-L-and-U-V-eq:1} and \eqref{lem:size-xi-in-L-and-U-V-eq:2} that
\[
x_v \leq \frac{r-2}{\lambda_1}+\frac{O(|V|)}{\lambda_1^2} = O (\frac{1}{\sqrt{n}}),
\]
and so, $u_0\in L$ since $x_{u_0} = 1$. 
    
(b) We consider an arbitrary vertex $v\in L\setminus\{u_0\}$.
By the eigen-equation with respect to $v$ and Lemma \ref{lem:bounds-lambda-1-n}, we have
\begin{align} \label{eq:lambda_1}
\lambda_1 (1 - x_v)
& = \lambda_1 (x_{u_0} - x_v) \nonumber\\
& = \sum_{u\in N(u_0)} x_u - \sum_{w\in N(v)} x_w  \nonumber\\
& = \sum_{u\in N(u_0)\setminus N(v)} x_u - \sum_{w\in N(v)\setminus N(u_0)} x_w  \nonumber\\
& \leq \sum_{u\in V} x_u + (r-3) x_{u_0}  \nonumber\\
& \leq |V|\cdot O (\frac{1}{\sqrt{n}}) + r-3~~~~~\text{(by Lemma \ref{lem:size-xi-in-L-and-U-V} (c))} \nonumber \\
& = O(1). ~~~~~~~~~~~~~~~~~~~~~~~~\text{(by  \eqref{lem:size-xi-in-L-and-U-V-eq:2} )}
\end{align}
The result immediately follows from Lemma \ref{lem:bounds-lambda-1-n} and \eqref{eq:lambda_1}. 
\end{proof}

Now, we are ready to show the main theorem of this section which indicate that  $G$ contains $r-2$ vertices that are adjacent to each of 
the rest of the $n-r+2$ vertices for $n$ is sufficiently large.

\begin{theorem}\label{thm:structure-kr-minor-free-max-spread}
Let $r\geq 3$, and $G$ be an $n$-vertex $K_r$-minor free graph with maximum spread. 
If $n$ is sufficiently large, we have $G= G[L]\vee (n - r + 2) K_1$, where $G[L]\subseteq K_{r-2}$.
\end{theorem}

\begin{proof}
By Lemma \ref{lem:L-U-V}, it suffices to show that $V = \emptyset$. Suppose that $|V| = t\geq 1$. 
By Theorem \ref{thm:edge-kr-minor-free}, we know that the average degree of $G[V]$ is less than $2C$. 
Thus, there is a vertex $u_1\in V$ such that $d_{G[V]}(u_1)\leq 2C-1$. 
By Lemma \ref{lem:L-U-V} (c), we can conclude that 
$d_{G[V \cup U]}(u_1) \leq 2C.$
Then deleting the vertex 
$u_1$ and the all edges in $G[V]$ incident to $u_1$. We can obtain the graph $G[V\setminus\{u_1\}]$. 
Clearly, $G[V\setminus \{u_1\}]$ is $K_r$-minor free and the average degree of $G[V\setminus\{u_1\}]$ 
is also less than $2C$. Hence there is $u_2\in V_2$, and $d_{G[V\setminus\{u_1\}]} (u_2)\leq 2C-1$.
Now, repeat the above process until the last vertex $u_t$ in $V$, then the vertices of $V$ can be 
rearranged by $u_1,\ldots,u_t$.
    
Let $G'$ be obtained from $G$ by removing the $2C$ edges (possibly less than $2C$ edges) in $G[V \cup U]  $ incident 
to $u_i\in V$ sequentially, and then connect $u_i$ to those vertices in $L$ that are not adjacent 
to $u_i$. It is easy to see $G'$ is also $K_r$-minor free, and $G'\subseteq K_{r-2}\vee (n - r + 2) K_1$.
    
Furthermore,  we firstly claim that $\lambda_1(G') > \lambda_1(G)$. Since
\begin{align*}
\bm{x}^\mathrm{T}\bm{x}\cdot\lambda_1(G')
& \geq \bm{x}^\mathrm{T} A(G') \bm{x} \\
& \geq \bm{x}^\mathrm{T} A(G) \bm{x}-2 \sum_{u_i\in V}  2C\cdot O(\frac{1}{\sqrt{n}})\cdot x_{u_i} + 2\sum_{u_i\in V} x_{u_i}\cdot \min\{x_{v_1},\ldots,x_{v_{r-2}}\} \\
& \geq \bm{x}^\mathrm{T} A(G) \bm{x}-2 \sum_{u_i\in V}  2C\cdot O(\frac{1}{\sqrt{n}})\cdot x_{u_i} + 2\sum_{u_i\in V}  x_{u_i} \big(1 - O(\frac{1}{\sqrt{n}})\big) \\
& > \bm{x}^\mathrm{T} A(G) \bm{x},
\end{align*}
where the second inequality and the third inequality  follow from Lemma \ref{lem:size-xi-in-L-and-U-V} (c) and (b), respectively. 

We also claim that $\lambda_n(G') < \lambda_n(G)$. In fact, let $\tilde{z}$ be the vector such 
that $\tilde{z}_u = z_u$ if $u\in L \cup U$ and $\tilde{z}_u = -|z_{u_i}|$  if $u\in V$. Then, by Lemma \ref{lem:large-degree} (c) and (d), we derive that
\begin{align*}
\tilde{\bm{z}}^\mathrm{T} \tilde{\bm{z}}\cdot\lambda_n(G')
& \leq \tilde{\bm{z}}^{\mathrm{T}} A(G')\tilde{\bm{z}} \\
& \leq \bm{z}^{\mathrm{T}} A(G)\bm{z} + 2\sum_{u_i\in V} \sum_{y\in N_{G[V\cup U]}(u_i)}|z_y z_{u_i}| 
- 2\sum_{u_i\in V} |z_{u_i}|\cdot \min\{z_{v_1},\ldots,z_{v_{r-2}}\} \\
& \leq \bm{z}^\mathrm{T} A(G)\bm{z} + 2\sum_{u_i\in V} 2C\cdot O(\frac{1}{\sqrt{n}})\cdot |z_{u_i}|
- 2\sum_{u_i\in V} |z_{u_i}|\cdot \big(1 - O (\frac{1}{\sqrt{n}})\big) \\
& < \bm{z}^{\mathrm{T}} A(G) \bm{z}.
\end{align*}
Therefore, $s(G') = \lambda_1(G') - \lambda_n(G') > \lambda_1(G) - \lambda_n(G) = s(G)$,  a contradiction. This completes the proof.
\end{proof}

\section{The extremal graph of $K_{r}$-minor free graphs with maximum spread}
\label{sec4}

In Section \ref{sec3}, we proved that for sufficiently large $n$, 
an $n$-vertex $K_{r}$-minor 
free graph $G$ with maximum spread admits a partition of $V(G) = L\cup U$ such that $G[U]$ is an independent set of size $n-r+2$. In this section, using 
the novel Laurent series expansion method developed by Li, Linz, Lu, and Wang \cite{LLLW}, we will  give the proof of Theorem \ref{thm:spread-kr-minor-free}.

\begin{lemma}\label{lem:spread-series-expansion}
Let $r\geq3$, and $G$ be the $n$-vertex $K_r$-minor free graph with maximum spread.
If $n$ is large enough, then the spread of $G$ satisfies 
\[
s(G) = 2\gamma_{n,r} +  \frac{1}{r-2}\left(-\frac{3}{4(r-2)}\ell_1^2+\ell_2\right) 
\cdot \frac{1}{\gamma_{n,r}} + O\left(\frac{1}{\gamma_{n,r}^3}\right),
\]
where
\begin{equation}\label{lem:spread-series-expansion-eq:1}
\ell_1 = \bm{1}^{\mathrm{T}} A(G[L]) \bm{1}, ~~
\ell_2 = \bm{1}^{\mathrm{T}} A(G[L])^2 \bm{1}.
\end{equation}
\end{lemma}

\begin{proof}
Let $\lambda$ be a non-zero eigenvalue of the adjacency matrix of $G$, and
let $\bm{y}$ be a normalized eigenvector corresponding to $\lambda$ such that $y_{v_{r-1}} = \cdots = y_{v_n} = 1$.
Let $A_L$ denote the adjacency matrix of the induced subgraph $G[L]$,
and let $\bm{y}_L$ denote the restriction of $\bm{y}$ to the vertices of $L$.

By Theorem \ref{thm:structure-kr-minor-free-max-spread}, the adjacency matrix of $G$ can be written as
\begin{equation}\label{eq:adj-mat-G}
A(G) = 
\begin{bmatrix}
A_L & J \\
J^{\mathrm{T}} & O 
\end{bmatrix}, ~~ 
J\in\mathbb{R}^{(r-2)\times (n-r+2)}.
\end{equation}
Since $\lambda$ is an eigenvalue of $A(G)$, we have
\[
\begin{bmatrix}
A_L & J \\
J^{\mathrm{T}} & O 
\end{bmatrix}
\begin{bmatrix}
\bm{y}_L \\ \bm{1}
\end{bmatrix}
= \lambda
\begin{bmatrix}
\bm{y}_L \\ \bm{1}
\end{bmatrix}.
\]
As a consequence,
\begin{equation}\label{lem:spread-series-expansion-eq:2}
A_L\bm{y}_L + (n-r+2) \cdot\bm{1} = \lambda\bm{y}_L.
\end{equation}
If $|\lambda| > \lambda_1(A_L)$, in light of \eqref{lem:spread-series-expansion-eq:2} we have
\begin{align}\label{lem:spread-series-expansion-eq:3}
\bm{y}_L 
& = (n-r+2) \cdot \left(\lambda I-A_L\right)^{-1} \bm{1} \nonumber \\
& = (n-r+2)\cdot \lambda^{-1} \left(I-\lambda^{-1} A_L\right)^{-1} \bm{1} \nonumber \\
& = (n-r+2)\cdot \lambda^{-1} \sum_{k=0}^\infty \left(\lambda^{-1} A_L\right)^k \bm{1} \nonumber \\
& = (n-r+2)\cdot \sum_{k=0}^\infty \lambda^{-(k+1)} A_L^k\bm{1}.
\end{align}
Here we use the assumption that $|\lambda| >\lambda_1(A_L)$ so that the infinite series converges.
By the eigen-equation at vertex $v_{r-1}$ and \eqref{lem:spread-series-expansion-eq:3}, we have
\begin{align}\label{lem:spread-series-expansion-eq:4}
\lambda = \lambda y_{v_{r-1}} = \bm{1}^{\mathrm{T}}\cdot\bm{y}_L
& = \bm{1}^{\mathrm{T}} \cdot (n-r+2)\cdot \sum_{k=0}^\infty \lambda^{-(k+1)} A_L^k \bm{1} \nonumber \\
& = (n-r+2) \cdot \sum_{k=0}^\infty \lambda^{-(k+1)} \bm{1}^{\mathrm{T}} A_L^k\bm{1} \nonumber \\
& = \frac{\gamma_{n,r}^2}{\lambda} + (n-r+2)\cdot\sum_{k=1}^\infty \lambda^{-(k+1)}\bm{1}^{\mathrm{T}} A_L^k\bm{1}.
\end{align}

Set $\ell_k := \bm{1}^{\mathrm{T}} A_L^k\bm{1}$ and $a_k := (n-r+2)\ell_k$ for $k\geq 0$.
Notices that $\ell_1 = \bm{1}^{\mathrm{T}} A_L\bm{1} = \sum_{v\in L} d_{G[L]}(v) = 2\left|E(G[L])\right|$ 
and $\ell_2=\bm{1}^{\mathrm{T}} A_L^2\bm{1} = \sum_{v\in L} d_{G[L]} (v)^2$.
It follows from \eqref{lem:spread-series-expansion-eq:4} that
\begin{equation}\label{eq:spread-series-expansion-eq:5}
\lambda^2 = \gamma_{n,r}^2 + \sum_{k=1}^\infty \frac{a_k}{\lambda^k}.
\end{equation}

According to Lemma \ref{lem:bounds-lambda-1-n}, both $\lambda_1(G)$ and $\lambda_n(G)$ satisfy 
Equation \eqref{eq:spread-series-expansion-eq:5}. The next claim demonstrates that these are the only eigenvalues that do so.

\begin{claim}\label{calm:bounds-lambda-i}
$|\lambda_i(G)| \leq \lambda_1(A_L)$ for $2 \leq i\leq n-1$.
\end{claim}

\begin{proof}
In view of \eqref{eq:adj-mat-G}, we see
\[
A(G) = 
\begin{bmatrix}
A_L & J \\
J^{\mathrm{T}} & O 
\end{bmatrix} 
= \begin{bmatrix}
A_L & O \\
O & O 
\end{bmatrix}
+ \begin{bmatrix}
O & J \\
J^{\mathrm{T}} & O 
\end{bmatrix}.
\]
Clearly, the eigenvalues of the matrix 
$\left[\begin{smallmatrix}
O & J \\
J^{\mathrm{T}} & O 
\end{smallmatrix}\right]$
are $\gamma_{n,r}$, $-\gamma_{n,r}$, and $0$ with multiplicity $n-2$.
By Weyl's inequalities \cite[Theorem 4.3.1]{Horn-Johnson2012}, for $i=2,\ldots,n-1$, we obtain
\[
\lambda_i(A(G)) \leq \lambda_1(A_L) 
+ \lambda_i\left(\left[
\begin{smallmatrix}
O & J \\
J^{\mathrm{T}} & O 
\end{smallmatrix}
\right]\right)
= \lambda_1(A_L),
\]
and
\[
\lambda_i(A(G)) \geq \lambda_n(A_L) + \lambda_i
\left(\left[
\begin{smallmatrix}
O & J \\
J^{\mathrm{T}} & O 
\end{smallmatrix}
\right]\right)
= \lambda_n(A_L) \geq -\lambda_1(A_L).
\]
Hence $|\lambda_i(G)|\leq |\lambda_1(A_L)|$ for $2 \leq i \leq n - 1$.
\end{proof}

By Lemma 27 in the appendix of \cite{LLLW}, $\lambda_1(G)$ has the following series expansion:
\[
\lambda_1(G) = \gamma_{n,r}+c_1+\frac{c_2}{\gamma_{n,r}} + O\left(\frac{1}{\gamma_{n,r}^2}\right),
\]
while $\lambda_n(G)$ is given by the series expansion:
\begin{align*}
\lambda_n(G) = -\gamma_{n,r} + c_1 - \frac{c_2}{\gamma_{n,r}} + O\left(\frac{1}{\gamma_{n,r}^2}\right).
\end{align*}
Here, $c_1$ and $c_2$ are determined from \eqref{eq:spread-series-expansion-eq:5} and  \cite[Lemma 11]{BGWLL} as 
\begin{align*}
c_1 & = \frac{a_1}{2\gamma_{n,r}^2} = \frac{\ell_1}{2(r-2)}, \\
c_2 & = -\frac{3}{8}\Big(\frac{a_1}{\gamma_{n,r}^2}\Big)^2 + \frac{1}{2} \frac{a_2}{\gamma_{n,r}^2} 
= -\frac{3}{8(r-2)^2} \ell_1^2 + \frac{1}{2(r-2)} \ell_2.
\end{align*}
Therefore, the spread of $G$ is
\begin{align*}
s(G) & = 2\gamma_{n,r} + \frac{2c_2}{\gamma_{n,r}} + O\left(\frac{1}{\gamma_{n,r}^3}\right) \\
& = 2\gamma_{n,r} +  \frac{1}{r-2}\left(-\frac{3}{4(r-2)}\ell_1^2+\ell_2\right) 
\cdot \frac{1}{\gamma_{n,r}} + O\left(\frac{1}{\gamma_{n,r}^3}\right).
\end{align*}

This completes the proof of Lemma \ref{lem:spread-series-expansion}.
\end{proof}

The following lemma is needed.

\begin{lemma}[\cite{Caen,Das}]\label{lem:Zagreb-index}
Let $H$ be a simple graph of order $n$ with $m$ edges. Then
\[
\sum_{i=1}^n{d(v_i)^2} \leq m \left( \frac{2m}{n-1} + n-2 \right),
\]
with equality holding if and only if $H \cong K_{1,n-1}$ or $H \cong K_n$ or $H \cong K_{n-1} \cup K_1$.
\end{lemma}

Now, we are ready to prove our main theorem.
\vspace{3mm}

\noindent {\it Proof of Theorem \ref{thm:spread-kr-minor-free}}.
Let $G$ be a graph achieving the maximum spread among all $K_r$-minor free graphs 
of order $n$. It suffices to show that $G = K_{r-2} \vee (n-r+2) K_1$. The assertion 
follows directly from Theorem \ref{thm:structure-kr-minor-free-max-spread} for $r=3$. 
From this point forward, we assume $r\geq 4$. Recall that 
$\ell_1 = \sum_{v\in V(L)} d_L(v)$ and $\ell_2 = \sum_{v\in V(L)} d_L(v)^2$. 
By \eqref{lem:spread-series-expansion-eq:1}, if $G[L] \cong K_{r-2}$, then 
\begin{equation}\label{thm:max-spread-kr-minor-free-eq:1}
c_2 = \frac{1}{2(r-2)} \cdot \left(\ell_2-\frac{3}{4(r-2)}\ell_1^2\right) = \frac{(r-3)^2}{8}.
\end{equation}

We now proceed by contradiction to show that $G[L] \cong K_{r-2}$. Suppose this is not the case.
Using \eqref{lem:spread-series-expansion-eq:1} and Lemma \ref{lem:Zagreb-index}, we obtain
\begin{align*}\label{valuec2}
c_2 & = \frac{1}{2(r-2)} \cdot \left(\ell_2-\frac{3}{4(r-2)}\ell_1^2\right) \\
& \leq \frac{1}{2(r-2)} \cdot \left(\frac{\ell_1}{2}\left(\frac{\ell_1}{r-3}+r-4\right) - \frac{3}{4(r-2)}\ell_1^2\right) \\
& = \frac{1}{2(r-2)} \cdot \left(\left(\frac{1}{2(r-3)}-\frac{3}{4(r-2)}\right)\ell_1^2 + \frac{r-4}{2}\ell_1 \right) \\
& = \frac{1}{2(r-2)} \cdot \left(\frac{-r+5}{4(r-2)(r-3)}\ell_1^2 + \frac{r-4}{2}\ell_1\right) \\
& < \frac{(r-3)^2}{8},
\end{align*}
where the last inequality holds because $G[L]$ is a proper subgraph of $K_{r-2}$,
and $$f(\ell_1):= \frac{-r+5}{4(r-2)(r-3)}\ell_1^2 + \frac{r-4}{2}\ell_1$$ is monotonically increasing as a function of $\ell_1$ over the interval $[0,+\infty)$ 
when $r\in\{4,5\}$, and over the interval $[0, (r-2)(r-3)(r-4)/(r-5)]$ when $r\geq 6$. Hence, 
$s(G) < s(K_{r-2} \vee (n-r+2) K_1)$ by Lemma \ref{lem:spread-series-expansion} and \eqref{thm:max-spread-kr-minor-free-eq:1}, 
a contradiction. 

This completes the proof of Theorem \ref{thm:spread-kr-minor-free}. \hfill $\Box$

\vspace{5mm}
\noindent{{\bf Acknowledgements}.}
The authors would like to thank Prof. Mingqing Zhai for his helpful discussions and comments.



\end{document}